\documentclass[10pt,reqno]{amsart}
\usepackage{amsmath,amssymb,amsthm,graphicx,epstopdf,mathrsfs,url}
\usepackage[usenames,dvipsnames]{color}
\usepackage{dsfont}   
\definecolor{darkred}{rgb}{0.7,0.1,0.1}
\definecolor{darkblue}{rgb}{0.1,0.1,0.4}
\definecolor{darkgrey}{rgb}{0.5,0.5,0.5}
\usepackage[colorlinks=true,linkcolor=darkred,citecolor=Blue]{hyperref}

\numberwithin{equation}{section}
\theoremstyle{plain}
\newtheorem{thm}{Theorem}[section]
\newtheorem{lem}[thm]{Lemma}

\theoremstyle{remark}

\theoremstyle{plain}

\newcommand{\hyp}[1]{$C^{2}$-hypersurface as in Definition~\ref{definition_hypersurface}}

\newcommand{\dom}{\mathrm{dom}\,}

\begin{document}
\title[]{Spectral transition for Dirac operators with electrostatic $\delta$-shell potentials supported on the straight line}


\author[J. Behrndt]{Jussi Behrndt}
\address{Institut f\"{u}r Angewandte Mathematik\\
Technische Universit\"{a}t Graz\\
 Steyrergasse 30, 8010 Graz, Austria\\
E-mail: {\tt behrndt@tugraz.at}}

\author[M. Holzmann]{Markus Holzmann}
\address{Institut f\"{u}r Angewandte Mathematik\\
Technische Universit\"{a}t Graz\\
 Steyrergasse 30, 8010 Graz, Austria\\
E-mail: {\tt holzmann@math.tugraz.at}}

\author[M. Tu\v{s}ek]{Mat\v{e}j Tu\v{s}ek}
\address{Department of Mathematics, Faculty of Nuclear Sciences and Physical Engeneering\\
Czech Technical University in Prague \\
Trojanova 13, 120 00, Prague \\
E-mail: {\tt matej.tusek@fjfi.cvut.cz}}


\maketitle

\begin{abstract}
  In this note the two dimensional Dirac operator $A_\eta$ with an electrostatic $\delta$-shell interaction of strength $\eta\in\mathbb R$ supported on a 
  straight line is studied. We observe a spectral transition in the sense that for 
  the critical interaction strengths $\eta=\pm 2$ 
  the continuous spectrum of $A_\eta$ inside the spectral gap of the free Dirac operator $A_0$ collapses abruptly to a single point.
\end{abstract}

\section{Introduction}

Differential operators that admit a spectral transition are of particular interest in mathematical analysis and its applications. 
Typically, one expects that the properties of a model described by a differential operator depend continuously on the parameters. However, in some cases it turns out that there is an abrupt change in the spectral properties -- in other words, a spectral transition. 
A well known example in this regard is the Smilansky model \cite{S04a, S04b} (see also \cite{EB14, EB17, EBKT16}) or the indefinite Laplacian 
studied in \cite{BK17, CPP19}.
A spectral transition was also observed for Dirac operators with singular potentials supported on bounded curves in $\mathbb{R}^2$ and surfaces in $\mathbb{R}^3$. More precisely, when studying  perturbations $A_\eta$ of the free Dirac operator by an electrostatic $\delta$-shell potential of strength $\eta \in \mathbb{R}$, it turned out that there is an abrupt change in the spectral properties for $\eta= \pm 2$. While for $\eta \neq \pm 2$ -- which is referred to as the {\it non-critical case} -- it is shown in the three dimensional situation in \cite{AMV14, AMV15, BEHL17, BEHL19_1} that the essential spectrum consists of two unbounded rays and finitely many eigenvalues between these rays, the {\it critical case} $\eta= \pm 2$ remained initially open. In the critical case it was then proved in \cite{BH19, B21, OV18} that there is a loss of smoothness in the operator domain and that there may be one additional point in the essential spectrum; for general combinations of interaction strengths in the two dimensional setting see \cite{BHOP20}. We note that similar effects also appear in the study of Dirac operators on bounded domains with suitable boundary conditions; cf. \cite{BHM20, BFSV17, B21, CLMT21, H21}.

In this note we study the spectrum of a Dirac operator in $\mathbb{R}^2$ with an electrostatic $\delta$-shell potential of strength $\eta \in \mathbb{R}$ supported on 
the straight line $\Sigma\cong\mathbb R$ which is formally given by
\begin{equation*}
  A_\eta = -i (\sigma_1 \partial_1 + \sigma_2 \partial_2) + m \sigma_3 + \eta \sigma_0 \delta_\Sigma;
\end{equation*}
here and in the following $\sigma_1, \sigma_2, \sigma_3$ are the $\mathbb{C}^{2 \times 2}$ valued Pauli spin matrices defined in~\eqref{def_Pauli_matrices}, $\sigma_0$ is the $2 \times 2$-identity matrix, and $m \in \mathbb{R} \setminus \{ 0 \}$. In order to define this expression rigorously, we
 denote by $\mathbb R^2_+$ and $\mathbb R^2_-$ the upper and lower half plane, respectively, and we 
use the notation $f_\pm := f\upharpoonright \mathbb R^2_\pm$ for the restriction of a function $f$ defined on $\mathbb{R}^2$. 
Next, let 
\begin{equation*}
  H(\sigma, \mathbb R^2_\pm) := \big\{ f \in L^2(\mathbb R^2_\pm; \mathbb{C}^2): (\sigma_1 \partial_1 + \sigma_2 \partial_2) f \in L^2(\mathbb R^2_\pm; \mathbb{C}^2) \big\}.
\end{equation*}
Then there exists a bounded Dirichlet trace operator  $H(\sigma, \mathbb R^2_\pm) \rightarrow H^{-1/2}(\Sigma; \mathbb{C}^2)$, 
$f_\pm\mapsto f_\pm\vert_\Sigma$ 
(one shows this in the same way as in \cite[Lemma~2.1]{BFSV17} and \cite[Lemma~2.3]{BFSV17} for bounded domains), and this
allows us to define for $\eta\in \mathbb{R}$ and $m \in \mathbb{R} \setminus  \{ 0 \}$ the operator
\begin{equation} \label{def_A_eta_tau_intro}
  \begin{split}
    A_\eta f &:= \big(-i (\sigma_1 \partial_1 + \sigma_2 \partial_2) + m \sigma_3 \big) f_+ \oplus \big(-i (\sigma_1 \partial_1 + \sigma_2 \partial_2) + m \sigma_3 \big) f_-, \\
    \dom A_\eta&:= \bigg\{ f = f_+ \oplus f_- \in H(\sigma, \mathbb R^2_+) \oplus H(\sigma, \mathbb R^2_-): \\
    &\qquad\qquad\qquad   i\sigma_2 (f_+|_\Sigma - f_-|_\Sigma) = \frac{\eta}{2}    (f_+|_\Sigma + f_-|_\Sigma) \bigg\},
  \end{split}
\end{equation}
in $L^2(\mathbb{R}^2; \mathbb{C}^2)$.
This operator is the rigorous mathematical definition of a Dirac operator with an electrostatic $\delta$-shell interaction of strength $\eta$ and models, for $m>0$, the propagation of a particle with mass $m$ and spin $1/2$ under the influence of such a potential; cf. \cite{AMV14, BEHL19_1, BHOP20, CLMT21}.
Observe that for $\eta=0$ the operator in \eqref{def_A_eta_tau_intro} coincides with the free Dirac operator 
\begin{equation} \label{def_free_op}
  A_0 f = -i (\sigma_1 \partial_1 + \sigma_2 \partial_2) f + m \sigma_3 f, \quad \dom A_0 = H^1(\mathbb{R}^2; \mathbb{C}^2),
\end{equation}
and recall that $A_0$ is self-adjoint in $L^2(\mathbb{R}^2; \mathbb{C}^2)$ with purely (absolutely) continuous spectrum
\begin{equation*}
  \sigma(A_0)  = (-\infty, -|m|] \cup [|m|, +\infty).
\end{equation*}
The main result of this paper is the following theorem on the self-adjointness and the spectra of the operators $A_\eta$. 
It turns out that the continuous spectrum grows
under the influence of the $\delta$-shell interaction and also half of the gap $(-|m|, |m|)$ is filled when $\eta$ approaches the critical values $\pm 2$
from above or below. In the critical case $\eta=\pm 2$ a spectral transition appears: the continuous spectrum inside $(-|m|, |m|)$ vanishes abruptly 
and $0$ becomes an infinite dimensional eigenvalue.

\begin{thm}\label{main}
  The operator $A_\eta$ is self-adjoint in $L^2(\mathbb{R}^2; \mathbb{C}^2)$ and the following holds for the spectrum of $A_\eta$:
  \begin{itemize}
      \item[(i)] If $\eta < -2$, then $\sigma(A_\eta)=\big(-\infty, -|m|\frac{\eta^2-4}{\eta^2+4}\big] \cup [|m|, +\infty)$.
      \item[(ii)] If $\eta =- 2$, then $\sigma(A_{\eta}) = (-\infty, -|m|] \cup \{ 0 \} \cup [|m|, +\infty)$.
      \item[(iii)] If $-2<\eta < 0$, then $\sigma(A_{\eta})=(-\infty, -|m|] \cup \big[|m| \frac{4-\eta^2}{\eta^2+4}, +\infty\big)$.
      \item[(iv)] If $\eta=0$, then $\sigma(A_{0})=(-\infty, -|m|] \cup \big[|m| , +\infty\big)$.
      \item[(v)] If $0<\eta < 2$, then $\sigma(A_{\eta})=\big(-\infty, -|m|\frac{4-\eta^2}{\eta^2+4}\big] \cup [|m|, +\infty)$.
      \item[(vi)] If $\eta = 2$, then $\sigma(A_{\eta}) = (-\infty, -|m|] \cup \{ 0 \} \cup [|m|, +\infty)$.
      \item[(vii)] If $2<\eta$, then $\sigma(A_{\eta})=(-\infty, -|m|] \cup \big[|m| \frac{\eta^2-4}{\eta^2+4}, +\infty\big)$.
    \end{itemize}
For $\eta\not=\pm 2$ the spectrum of $A_\eta$ is purely continuous
  and for $\eta=\pm 2$ the point $0$ is an isolated eigenvalue of $A_\eta$ with infinite multiplicity 
  and the remaining spectrum 
  is purely continuous. 
\end{thm}

Note that the spectrum of $A_\eta$ is invariant under the transformation $\eta \mapsto -\frac{4}{\eta}$. This symmetry would also follow from the stronger fact that $A_\eta$ and $A_{-4/\eta}$ are unitary equivalent; this can be shown in the same way as in \cite[Proposition~4.8~(i) and Proposition~4.15~(i)]{BHOP20}.

The proof of Theorem~\ref{main} is based on an efficient abstract technique that was applied in a similar form 
also in \cite{BHOP20, CLMT21}: We use a so-called boundary triple and its Weyl function 
to reduce the spectral analysis of $A_\eta$ in the gap $(-\vert m\vert,\vert m\vert)$ of $A_0$ 
to a certain boundary operator in $L^2(\Sigma; \mathbb{C}^2) \simeq L^2(\mathbb{R}; \mathbb{C}^2)$; cf. 
\cite{BHS19, BGP, DM91, DM95} for details on boundary triples and Weyl functions in the extension theory of symmetric operators. 
Since $\Sigma$ 
is a straight line the spectral properties of this boundary operator can be studied with the help of the Fourier transform. 
From the limit behaviour of the Weyl function and the boundary operator towards the real line we also conclude
that the set $(-\infty, -|m|] \cup [|m|, +\infty)$ consists of purely continuous spectrum of $A_\eta$ for all $\eta\in\mathbb R$. For $m=0$ this method is also applicable, then it follows for all $\eta \in \mathbb{R}$ that $\sigma(A_\eta) = \mathbb{R}$, i.e. there is no spectral transition.
We note that the method of direct integrals decomposing $A_\eta$ in its fibers would yield a similar result; 
cf. \cite{GL06} for a related problem on Dirac operators with Robin type boundary conditions. 
In this note we prefer to work with the boundary triple technique, since it allows generalizations for more complicated curves $\Sigma$ in a natural way, which we plan to study in the near future.

\subsection*{Notations}

Let 
\begin{equation} \label{def_Pauli_matrices}
   \sigma_1 := \begin{pmatrix} 0 & 1 \\ 1 & 0 \end{pmatrix}, \qquad
   \sigma_2 := \begin{pmatrix} 0 & -i \\ i & 0 \end{pmatrix}, \qquad
   \sigma_3 := \begin{pmatrix} 1 & 0 \\ 0 & -1 \end{pmatrix},
\end{equation}
be the Pauli spin matrices and denote by $\sigma_0$ the $2 \times 2$-identity matrix. Note that the Pauli matrices satisfy 
$ \sigma_j \sigma_k + \sigma_k \sigma_j = 2 \delta_{jk} \sigma_0$, $j,k \in \{ 1,2,3\}$.
For $x = (x_1,x_2)$ we will often write $\sigma \cdot x = \sigma_1 x_1 + \sigma_2 x_2$ and $\sigma \cdot \nabla = \sigma_1 \partial_1 + \sigma_2 \partial_2$.

    \subsection*{Acknowledgements}
    J. Behrndt and M. Holzmann gratefully acknowledge financial support by the Austrian Science Fund (FWF): P 33568-N. M.~Tu\v{s}ek was partially supported by the grant No.~21-07129S of the Czech Science Foundation (GA\v{C}R) and by the project CZ.02.1.01/0.0/0.0/16\_019/0000778 from the European Regional Development Fund. This publication is based upon work from COST Action CA 18232 MAT-DYN-NET, supported by COST (European Cooperation in Science and Technology), www.cost.eu.

 \section{Proof of Theorem~\ref{main}}

 Let $A_0$ be the free Dirac operator in \eqref{def_free_op} and recall that for $z \in \rho(A_0)$ the resolvent of $A_0$ is given by
\begin{equation*}
  \big[(A_0 - z)^{-1} f\big](x) = \int_{\mathbb{R}^2} G_z(x - y) f(y) d y, \quad f \in L^2(\mathbb{R}^2; \mathbb{C}^2),\, x \in \mathbb{R}^2,
\end{equation*}
where
\begin{equation*} 
  G_z(x) = \frac{i\sqrt{m^2 - z^2}}{2 \pi} K_1 \big(\sqrt{m^2 - z^2} | x |\big)  \frac{\sigma \cdot x}{ | x | }  + \frac{1}{2 \pi} K_0 \big(\sqrt{m^2 - z^2} | x |\big) \big( z \sigma_0 + m \sigma_3\big)
\end{equation*}
and $K_j$ is the modified Bessel functions of the second kind and order $j$; cf. \cite{AS84} and \cite{T92}. Here and in the following the complex square root is chosen such that it is holomorphic in $\mathbb{C} \setminus (-\infty, 0]$ and $\text{Re}\, \sqrt{z} > 0$. An important object in our analysis is the mapping $\mathcal{C}_z$ which is defined for $z \in \rho(A_0)$ by
\begin{equation*}
  \mathcal{C}_z \varphi := G_z * \varphi, \quad \varphi \in \mathcal{S}(\Sigma; \mathbb{C}^2),
\end{equation*}
where $\mathcal{S}(\Sigma; \mathbb{C}^2) \simeq \mathcal{S}(\mathbb{R}; \mathbb{C}^2)$ denotes the Schwartz space and the convolution is understood in the sense of distributions.
The action of $\mathcal{C}_z$ is 
\begin{equation*}
  \mathcal{C}_z \varphi(x) := \lim_{\varepsilon \searrow 0} \int_{\mathbb{R} \setminus (x_1-\varepsilon, x_1+\varepsilon)} G_z(x-y) \varphi(y) d y_1, \quad x=(x_1,0), y=(y_1,0) \in \Sigma.
\end{equation*}
In the following we will denote by $\mathcal{F}$ the Fourier transform on $\Sigma \simeq \mathbb{R}$.

\begin{lem} \label{proposition_C_z_line}
  The map $\mathcal{F} \mathcal{C}_z \mathcal{F}^{-1}$ is the multiplication operator with the matrix valued function
  \begin{equation*}
     p\mapsto \begin{pmatrix} \frac{z + m}{2 \sqrt{p^2+m^2-z^2}} & \frac{p}{2 \sqrt{p^2+m^2-z^2}} \\ \frac{p}{2 \sqrt{p^2+m^2-z^2}} & \frac{z - m}{2 \sqrt{p^2+m^2-z^2}} \end{pmatrix}.
  \end{equation*}
  In particular, $\mathcal{C}_z$ gives rise to a bounded operator in $H^s(\Sigma; \mathbb{C}^2)$ for any $s \in \mathbb{R}$.
\end{lem}

\begin{proof}
  Let $\phi_z(x_1) := \frac{1}{2 \pi} K_0(\sqrt{m^2-z^2}|x_1|)$, $x_1 \in \mathbb{R} \setminus \{ 0 \}$. Since the Fourier transform takes $K_0(\kappa |x|)$ to
  $\sqrt{\pi/2}(p^2+\kappa^2)^{-1/2}$ we obtain
  $$\mathcal{F} (\phi_z * \varphi)(p) = \frac{1}{2 \sqrt{p^2+m^2-z^2} } \mathcal{F} \varphi(p),\quad  \varphi \in \mathcal{S}(\mathbb{R}).$$ 
  Now observe that
  \begin{equation*}
    G_z(x) = \left( -i \sigma_1 \frac{d}{dx_1} + m \sigma_3 + z \sigma_0 \right) \phi_z(x_1),\quad x=(x_1,0) \in \Sigma,
  \end{equation*}
  and hence the calculation rules for the Fourier transform of distributions \cite[Chapter~IX]{RSII} lead to
  \begin{equation*}
    \begin{split}
      \mathcal{F} \mathcal{C}_z \mathcal{F}^{-1} \varphi(p) &= \left( \sigma_1 p + m \sigma_3 + z \sigma_0 \right) \frac{1}{2 \sqrt{p^2+m^2-z^2}} \varphi(p) \\
      &=  \begin{pmatrix} \frac{z + m}{2 \sqrt{p^2+m^2-z^2}} & \frac{p}{2 \sqrt{p^2+m^2-z^2}} \\ \frac{p}{2 \sqrt{p^2+m^2-z^2}} & \frac{z - m}{2 \sqrt{p^2+m^2-z^2}} \end{pmatrix} \varphi(p)
    \end{split}
  \end{equation*}
  for $\varphi \in \mathcal{S}(\Sigma; \mathbb{R}^2)$,
  which yields the claimed result about the representation of $\mathcal{F} \mathcal{C}_z \mathcal{F}^{-1}$. Finally, taking the definition of the norm in $H^s(\Sigma; \mathbb{C}^2) \simeq H^s(\mathbb{R}; \mathbb{C}^2)$ with the help of the Fourier transform into account, one sees that $\mathcal{C}_z$ gives rise to a bounded operator in $H^s(\Sigma; \mathbb{C}^2)$ for all $s \in \mathbb{R}$.
\end{proof}

In the following
we shall make use of the closed symmetric restriction 
\begin{equation} \label{def_S} 
  S f = (-i \sigma \cdot \nabla + m \sigma_3) f, \quad \dom S = H^1_0(\mathbb{R}^2 \setminus \Sigma; \mathbb{C}^2),
\end{equation}
of the free Dirac operator $A_0$,
and its adjoint
\begin{equation*} 
  \begin{split}
    S^* f &= (-i \sigma \cdot \nabla + m \sigma_3) f_+ \oplus (-i \sigma \cdot \nabla + m \sigma_3) f_-, \\
    \dom S^* &= H(\sigma, \mathbb R^2_+) \oplus H(\sigma, \mathbb R^2_-);
  \end{split}
\end{equation*}
the above mentioned properties of $S$ and $S^*$ can be shown in the same way as in \cite[Proposition~3.1]{BH19}, where similar operators in $\mathbb{R}^3$ with compact surfaces $\Sigma$ have been studied.
 
Fix some $\zeta \in \rho(A_0)$. Define the maps $\Gamma_0, \Gamma_1: \dom S^* \rightarrow L^2(\Sigma; \mathbb{C}^2)$ by
  \begin{equation}\label{bt}
    \begin{split}
      \Gamma_0 f &:= -i \Lambda^{-1}  \sigma_2 \big(f_+\vert_\Sigma - f_-\vert_\Sigma \big), \\
      \Gamma_1 f &:= \frac{1}{2} \Lambda  \big( (f_+\vert_\Sigma + f_-\vert_\Sigma) - (\mathcal{C}_\zeta + \mathcal{C}_{\overline{\zeta}})  \Lambda \Gamma_0 f \big), \quad f = f_+ \oplus f_- \in \dom S^*,
    \end{split}
  \end{equation}
  where $\Lambda=(-\Delta+1)^\frac{1}{4}=\mathcal F^{-1}(p^2+1)^\frac{1}{4} \mathcal{F}$ is viewed as a bijective 
  operator from $H^s(\Sigma)$ onto $H^{s-1/2}(\Sigma)$ for $s=1/2$ or $s=0$,
  or as a self-adjoint unbounded operator in $L^2(\Sigma)$ defined on $H^{1/2}(\Sigma)$.
  It follows in the same way as in \cite[Proposition~3.5]{BHOP20} or \cite[Proposition~6.1]{CLMT21} that 
  $\{ L^2(\Sigma; \mathbb{C}^2), \Gamma_0, \Gamma_1 \}$ is a boundary triple for $S^*$ with $A_0 = S^* \upharpoonright \ker\Gamma_0$ and corresponding Weyl function 
  \begin{equation*}
    \rho(A_0) \ni z \mapsto M(z) = \Lambda  \left( \mathcal{C}_z - \frac{1}{2} (\mathcal{C}_\zeta + \mathcal{C}_{\overline{\zeta}}) \right)  \Lambda;
  \end{equation*}
  cf. \cite{BHS19, BGP, DM91, DM95} for details on boundary triples and Weyl functions in the extension theory of symmetric operators. 

Next we describe the operators $A_\eta$ with the help of the boundary triple \eqref{bt} and also conclude their self-adjointness in $L^2(\mathbb R^2; \mathbb{C}^2)$. 
The proof of this result follows the lines of \cite[Proposition~6.3]{CLMT21} and \cite[Proposition~4.3]{BHOP20}. 

\begin{lem} \label{proposition_A_eta_tau_bt}
 For $\eta\in \mathbb{R}\setminus\{0\}$ the operator
    \begin{equation*} 
      \begin{split}
        \Theta \varphi &:= -\Lambda \left[ \frac{1}{\eta}\sigma_0   + \frac{1}{2}  (\mathcal{C}_\zeta + \mathcal{C}_{\overline{\zeta}})  \right] \Lambda \varphi, \\
        \dom \Theta &= \bigg \{ \varphi \in L^2(\Sigma; \mathbb{C}^2): \left[ \frac{1}{\eta}\sigma_0  + \frac{1}{2}  (\mathcal{C}_\zeta + \mathcal{C}_{\overline{\zeta}})  \right] \Lambda \varphi \in H^{1/2}(\Sigma; \mathbb{C}^2) \bigg\},
      \end{split}
    \end{equation*}
    is self-adjoint in $L^2(\Sigma; \mathbb{C}^2)$ and we have 
    \begin{equation}\label{aaa}
    A_\eta = S^* \upharpoonright \ker(\Gamma_1 - \Theta \Gamma_0).
    \end{equation}
    In particular, the operator $A_\eta$ is self-adjoint in $L^2(\mathbb{R}^2; \mathbb{C}^2)$.
\end{lem}

\begin{proof}
  Define the $\mathbb{C}^{2 \times 2}$-valued function
  \begin{equation*}
    \theta(p) := -\sqrt{p^2+1}  \begin{pmatrix} \frac{1}{\eta} + \textup{Re} \, \frac{\zeta + m}{2 \sqrt{p^2+m^2-\zeta^2}} & \textup{Re} \, \frac{p}{2 \sqrt{p^2+m^2-\zeta^2}} \\ \textup{Re} \, \frac{p}{2 \sqrt{p^2+m^2-\zeta^2}} & \frac{1}{\eta } + \textup{Re} \, \frac{\zeta - m}{2 \sqrt{p^2+m^2-\zeta^2}} \end{pmatrix}.
  \end{equation*}
  Then, by Lemma~\ref{proposition_C_z_line} we have 
  \begin{equation*} 
    \begin{split}
      \mathcal{F} \Theta  \mathcal{F}^{-1} \varphi (p) &= \theta(p) \varphi(p), \\
      \dom \mathcal{F}  \Theta \mathcal{F}^{-1} &= \{ \varphi \in L^2(\mathbb{R}; \mathbb{C}^2): \theta \varphi \in L^2(\mathbb{R}; \mathbb{C}^2) \}.
    \end{split}
  \end{equation*}
  Since $\theta$ is a symmetric matrix and $\mathcal{F}$ is unitary, we conclude that $\Theta$ is self-adjoint in $L^2(\Sigma; \mathbb{C}^2)  \simeq L^2(\mathbb{R}; \mathbb{C}^2)$.   Using \eqref{bt} is not difficult to verify that \eqref{aaa} holds and hence the self-adjointness of $A_\eta$ follows, see, e.g., 
  \cite[Corollary 2.1.4~(v)]{BHS19}.
\end{proof}

In the next lemma we analyze, when zero belongs to the point spectrum or continuous spectrum of $\Theta - M(z)$ for
$z \in (-|m|, |m|)$. 

\begin{lem} \label{proposition_essential_spectrum_Theta}
  Let $\eta\in \mathbb{R}\setminus\{0\}$. For $z \in (-|m|, |m|)$ the following holds:
  \begin{itemize}
    \item[(i)] If $\eta\neq \pm 2$, then $0 \notin \sigma_\textup{p}(\Theta - M(z))$ and $0 \in \sigma_\textup{c}(\Theta - M(z))$ if and only if 
    \begin{equation} \label{essential_spectrum_Theta1}
      \frac{z \eta}{\eta^2-4}>0
    \end{equation}
    and
    \begin{equation*} 
      z \leq  - \frac{  |m (\eta^2-4) |}{\eta^2+4} \quad\text{or}\quad z \geq   \frac{  |m (\eta^2-4) |}{\eta^2+4}.
    \end{equation*}
    \item[(ii)] If $\eta=\pm 2$, then $0 \in \sigma_\textup{p}(\Theta - M(0))$ with $\dim\ker(\Theta - M(0))=+\infty$,
    and $0 \in \rho(\Theta - M(z))$ for $z \neq 0$.
  \end{itemize}
\end{lem}
\begin{proof}
  Let $z \in (-|m|, |m|)$ and observe that
  \begin{equation} \label{unitary_equivalent} 
        \Theta-M(z) = -\Lambda \left[ \frac{1}{\eta}\sigma_0   + \mathcal{C}_z  \right] \Lambda \varphi, \quad\dom (\Theta-M(z))=\dom\Theta.
    \end{equation}
   Using Lemma~\ref{proposition_C_z_line} we conclude   
   \begin{equation*}
    \begin{split}
      \mathcal{F} (\Theta-M(z))  \mathcal{F}^{-1} \varphi (p) &= \theta_z(p) \varphi(p), \\
      \dom \mathcal{F}  (\Theta-M(z))\mathcal{F}^{-1} &= \{ \varphi \in L^2(\mathbb{R}; \mathbb{C}^2): \theta_z \varphi \in L^2(\mathbb{R}; \mathbb{C}^2) \},
    \end{split}
  \end{equation*} 
   where 
   \begin{equation*}
    \theta_{z}(p)=-\sqrt{p^2+1}  \begin{pmatrix} \frac{1}{\eta} + \frac{z + m}{2 \sqrt{p^2+m^2-z^2}} & \frac{p}{2 \sqrt{p^2+m^2-z^2}} \\ \frac{p}{2 \sqrt{p^2+m^2-z^2}} & \frac{1}{\eta} + \frac{z - m}{2 \sqrt{p^2+m^2-z^2}} \end{pmatrix},
  \end{equation*}
  and hence it suffices to consider the self-adjoint multiplication operator with the function $\theta_z$ in $L^2(\mathbb{R}; \mathbb{C}^2)$.
  In the following we discuss for which $\eta\in \mathbb{R}\setminus\{0\}$ and $z\in (-|m|, |m|)$ the multiplication operator $\theta_z$ 
  has $0$ as an eigenvalue or as a point in the continuous spectrum. Note first that
  $$
  \det \theta_z(p) = (p^2+1) \left[ \frac{1}{\eta^2}+ \frac{z}{\eta\sqrt{p^2+m^2-z^2}}-\frac{1}{4}\right].
  $$
  
We verify assertion (ii).
 In the case $\eta=\pm 2$ we have $\det \theta_z(p)=0$ if and only if $z=0$, and in this situation $\det \theta_z(p)=0$ for all $p\in\mathbb R$.
 Therefore, $0$ is an eigenvalue of infinite multiplicity of the multiplication operator $\theta_z$. Observe that for $z \in (-|m|, |m|)\setminus\{0\}$
 we have $\det \theta_z(p)\not=0$ and hence $0$ is in the resolvent set of $\theta_z$. 
 
 Now we prove (i).
 If $\eta\neq \pm 2$, 
 then $\det \theta_z(p)=0$ is equivalent to
  \begin{equation} \label{zero_equation2}
    \begin{split}
      \sqrt{p^2+m^2-z^2} = 4 \frac{z \eta}{\eta^2  - 4}.
    \end{split}
  \end{equation}
  The left hand side of the last equation is positive and hence there exist solutions only if also the right hand side is positive, i.e. only if~\eqref{essential_spectrum_Theta1} holds. Assuming this, we see by squaring the last equation that it is equivalent to
  \begin{equation}\label{lolo}
    p^2 = \frac{(\eta^2  + 4)^2}{(\eta^2  - 4)^2}\,z^2  - m^2.
  \end{equation}
  When $p \in \mathbb{R}$ varies the left hand side can be any non-negative number, and hence \eqref{zero_equation2} has a solution, whenever the right hand side is non-negative, i.e. if and only if $z \leq z_-$ or $z \geq z_+$, where $z_{\pm}$ are the zeros of the polynomial on the right hand side of the last equation given by
  \begin{equation*}
    z_{\pm} = \pm \left| m \frac{\eta^2 -4}{\eta^2 +4}\right|.
  \end{equation*}
  It is also clear that for a fixed $z \in (-|m|, |m|)$ with $z \leq z_-$ or $z \geq z_+$ there are only two values $p\in\mathbb R$
  such that \eqref{lolo} holds, and hence $0$ is in the continuous spectrum of the multiplication operator associated with $\theta_z$.
\end{proof}

In order to prove Theorem~\ref{main}, we also investigate the limiting behaviour of $( \Theta - M(z) )^{-1}$, when $z \in \mathbb{C} \setminus \mathbb{R}$ approaches $\sigma(A_0) = (-\infty, -|m|] \cup [|m|, +\infty)$:

\begin{lem} \label{Proposition_continuous_spectrum}
  Let $\eta \in \mathbb{R} \setminus \{ 0 \}$. Then, the following is true:
  \begin{itemize}
    \item[(i)] For any $x \in (-\infty, -|m|] \cup [|m|, +\infty)$ and all $\varphi \in L^2(\Sigma; \mathbb{C}^2)$ one has  
    \begin{equation} \label{limit1}
      \lim_{y \searrow 0} i y \big( \Theta - M(x+iy) \big)^{-1} \varphi = 0.
    \end{equation}
    \item[(ii)] For any $x \in (-\infty, -|m|) \cup (|m|, +\infty)$ there exists $\varphi \in L^2(\Sigma; \mathbb{C}^2)$ such that
    \begin{equation} \label{limit2}
      \lim_{y \searrow 0} \textup{Im}\, \big( \big( \Theta - M(x+iy) \big)^{-1} \varphi, \varphi \big) \neq 0.
    \end{equation}
  \end{itemize}
\end{lem}
\begin{proof}
In order to show the claims, we note first with the help of~\eqref{unitary_equivalent} that the operator $\mathcal{F} (\Theta-M(z))^{-1}  \mathcal{F}^{-1}$, $z \in \mathbb{C} \setminus \mathbb{R}$, is the maximal multiplication operator associated with the matrix-valued function
\begin{equation*}
  \begin{split} 
    &\theta_{z}^{-1}(p)= -\frac{\sqrt{p^2+1}}{\det \theta_z(p)} \begin{pmatrix} \frac{1}{\eta} + \frac{z - m}{2 \sqrt{p^2+m^2-z^2}} & -\frac{p}{2 \sqrt{p^2+m^2-z^2}} \\ -\frac{p}{2 \sqrt{p^2+m^2-z^2}} & \frac{1}{\eta} + \frac{z + m}{2 \sqrt{p^2+m^2-z^2}} \end{pmatrix} \\
    &~~= -\frac{2 \eta}{c_z(p) \sqrt{p^2+1}} \begin{pmatrix} 2 \sqrt{p^2+m^2-z^2} + \eta(z - m) & -\eta p \\ -\eta p & 2 \sqrt{p^2+m^2-z^2} + \eta(z + m) \end{pmatrix}
  \end{split}
\end{equation*}
with
\begin{equation*}
  c_z(p) = (4-\eta^2)\sqrt{p^2+m^2-z^2} + 4\eta z.
\end{equation*}
By the continuity of the complex square root one sees for a fixed $p \in \mathbb{R}$ that the limit $c_{x+i0}(p) := \lim_{y \searrow 0} c_{x+iy}(p)$ exists.
One verifies for any $x \in (-\infty, -|m|] \cup [|m|, +\infty)$ in a similar way as in~\eqref{zero_equation2} and~\eqref{lolo} that $c_{x+i0}(p)$ has no zero, if $\eta = \pm 2$ or if $\frac{x \eta}{\eta^2-4}<0$, and for $\frac{x \eta}{\eta^2-4}>0$ the term $c_{x+i0}(p)$ has two zeros at
\begin{equation} \label{p_pm}
  p_\pm = \pm \left(  \frac{(\eta^2  + 4)^2}{(\eta^2  - 4)^2}\,x^2  - m^2 \right)^{1/2}.
\end{equation}

Let us prove~(i). Let $x \in (-\infty, -|m|] \cup [|m|, +\infty)$ be fixed. Since $z\mapsto M(z)$
is the Weyl function of a boundary triple it is a Nevanlinna function
and the values $M(z)$ are bounded and everywhere defined operators in 
$L^2(\Sigma; \mathbb{C}^2)$; cf. \cite[Corollary~2.3.7]{BHS19}. It follows
that $z\mapsto (\Theta-M(z))^{-1}$ is also a Nevanlinna function and for 
$z\in\rho(A_\eta)\cap\rho(A_0)$ the values are also bounded and everywhere defined operators in 
$L^2(\Sigma; \mathbb{C}^2)$. From the operator representation  of Nevanlinna functions (see, e.g., 
\cite[Theorem~4.2]{HSW98}) we then conclude
that there exists a constant $C_1 > 0$ such that
\begin{equation*}
  \big\| (\Theta-M(x+iy))^{-1} \big\| \leq \frac{C_1}{y}
\end{equation*}
for $y>0$ sufficiently small.
Since $(\Theta-M(x+iy))^{-1}$ is unitary equivalent to the multiplication operator with the function $\theta_{x+iy}^{-1}$, we conclude that
\begin{equation} \label{estimate_theta1}
  \big\| \theta_{x+iy}^{-1} \big\|_\infty \leq \frac{C_1}{y}
\end{equation}
for all $y>0$ sufficiently small.
Next, define for $y>0$ the set
\begin{equation} \label{def_I}
  I_x(y) := \begin{cases} \{ p \in \mathbb{R}: |p-p_\pm| \geq \sqrt{y} \}, & \text{if } c_{x+i0} \text{ has a zero}, \\ \mathbb{R},  & \text{if } c_{x+i0} \text{ has no zero}, \end{cases}
\end{equation}
and prove for all sufficiently small $y>0$ and $p \in I_x(y)$ that
\begin{equation} \label{estimate_theta2}
  \big| \theta_{x+iy}^{-1}(p) \big| \leq \frac{C_2}{\sqrt{y}}
\end{equation}
for some $C_2 > 0$, where the absolute value is understood element wise. 

In order to show~\eqref{estimate_theta2}, we note first that for some $C_3>0$ independent of $p \in \mathbb{R}$ and $y \in [0,1]$ the estimate
\begin{equation} \label{estimate_theta3}
  \big| c_{x+iy}(p) \theta_{x+iy}^{-1}(p) \big| \leq C_3
\end{equation}
holds. 

Consider first the case when $c_{x+i0}$ has no zero. From the definition of $c_{x+iy}$ we conclude that there exists $P_0 > 0$ such that $\big| c_{x+iy}(p) \big| > 1$ holds for all $|p| > P_0$ and all sufficiently small $y>0$.
Moreover, the map $[-P_0, P_0] \times [0, 1] \ni (p,y) \mapsto c_{x+iy}(p)$ is uniformly continuous and hence, there exists a constant $C_4 > 0$ such that for all sufficiently small $y > 0$ and $p \in [-P_0, P_0]$ the relation $|c_{x+iy}(p)| > C_4$ holds, as $c_{x+i0}$ has no zero. Thus, for $y>0$ sufficiently small $|c_{x+iy}|$ is uniformly bounded from below by a positive constant, which implies with~\eqref{estimate_theta3} the estimate in~\eqref{estimate_theta2}.

Assume now that $c_{x+i0}$ has the zeros $p_\pm$, let $y>0$ be sufficiently small, and fix $p \in \mathbb{R}$ with $|p-p_\pm| \geq \sqrt{y}$. From~\eqref{p_pm} we get $\sqrt{p_\pm^2 + m^2 - x^2} > 0$ and thus,
\begin{equation*}
  \begin{split}
    &\bigg|\sqrt{p^2+m^2-(x+iy)^2} - \sqrt{p_\pm^2+m^2-x^2}\bigg| \\
    &\qquad= \left| \frac{p^2-(x+iy)^2 - p_\pm^2 + x^2}{\sqrt{p^2+m^2-(x+iy)^2} + \sqrt{p_\pm^2+m^2-x^2}} \right| \\
    &\qquad \geq \left| \frac{(p-p_\pm)(p + p_\pm)}{\sqrt{p^2+m^2-(x+iy)^2} + \sqrt{p_\pm^2+m^2-x^2}} \right| - \left| \frac{x^2-(x+iy)^2}{\sqrt{p_\pm^2+m^2-x^2}} \right| \\
    & \qquad \geq C_5 \sqrt{y}
  \end{split}
\end{equation*}
with a constant $C_5>0$, where we used in the first inequality for the denominator that the real part of the complex square root is positive and in the last estimate that  the functions
\begin{equation*}
  \mathbb{R}_{\pm} \ni p \mapsto \left|\frac{p \pm  |p_\pm|}{\sqrt{p^2+m^2-(x+iy)^2} + \sqrt{p_\pm^2+m^2-x^2}} \right|
\end{equation*}
are uniformly bounded from below independently of $y>0$ sufficiently small. Hence, we conclude for those $y>0$ that 
\begin{equation*}
  \big| c_{x+iy}(p)\big| = \big| c_{x+iy}(p)-c_{x+i0}(p_\pm) \big| \geq C_6 \sqrt{y}.
\end{equation*}
Together with~\eqref{estimate_theta3} this yields~\eqref{estimate_theta2}.

Now, we have everything in hands to prove item~(i).
Let $\chi$ be the characteristic function for the set $I_x(y)$ defined in~\eqref{def_I}. Then~\eqref{estimate_theta1} and~\eqref{estimate_theta2} imply for an arbitrary $\varphi \in L^2(\Sigma; \mathbb{C}^2)$
\begin{equation*}
 \begin{split}
  \lim_{y \searrow 0} \big\| y\theta_{x+iy}^{-1} \varphi \big\|_{L^2(\mathbb{R}; \mathbb{C}^2)} &\leq \lim_{y \searrow 0} \big\| y (1-\chi) \theta_{x+iy}^{-1} \varphi \big\|_{L^2(\mathbb{R}; \mathbb{C}^2)} + \lim_{y \searrow 0} \big\| y \chi\theta_{x+iy}^{-1} \varphi \big\|_{L^2(\mathbb{R}; \mathbb{C}^2)} \\
  &\leq \lim_{y \searrow 0} C_1 \big\| (1-\chi) \varphi \big\|_{L^2(\mathbb{R}; \mathbb{C}^2)} + \lim_{y \searrow 0} C_2 \sqrt{y} \big\| \varphi \big\|_{L^2(\mathbb{R}; \mathbb{C}^2)} =0.
\end{split}
\end{equation*}
Thus, assertion~(i) is true.

To show item~(ii), let $x \in (-\infty, -|m|) \cup(|m|, +\infty)$ and let $I$ be a non-empty compact interval which is contained in $(-\sqrt{x^2-m^2},\sqrt{x^2-m^2})$. Note that the numbers $p_\pm$ in~\eqref{p_pm} are not contained in $I$ and that 
\begin{equation*}
  \lim_{y \searrow 0} \text{Im}\, \sqrt{p^2 + m^2 - (x+iy)^2} \neq 0
\end{equation*}
for $p \in I$. Let $\widetilde{\chi}$ be the characteristic function for $I$. Taking the form of $\theta_{x+iy}^{-1}$ into account, we get that $\text{Im}\, (\widetilde{\chi} \theta_{x+iy}^{-1})$ converges uniformly to a nontrivial limit and hence, one can choose $\varphi \in L^2(\Sigma; \mathbb{C}^2)$ supported in $I$ such that
\begin{equation*}
  \lim_{y \searrow 0} \textup{Im}\, \big( \big(  \Theta - M(x+iy) \big)^{-1} \varphi, \varphi \big) \neq 0,
\end{equation*}
i.e. also item~(ii) is true.
\end{proof}

\begin{proof}[Proof of Theorem~\ref{main}]
First, we note that the case $\eta=0$ reduces to the free Dirac operator and hence, all claims are known here. So we assume from now on that $\eta \neq 0$. The self-adjointness of $A_\eta$ is shown in Lemma~\ref{proposition_A_eta_tau_bt}. The claims about $\sigma(A_\eta) \cap (-|m|, |m|)$ follow from Lemma~\ref{proposition_essential_spectrum_Theta} and \cite[Theorem~2.6.2]{BHS19}.

It remains to show for $\eta \in \mathbb{R} \setminus \{ 0 \}$ that $(-\infty, -|m|] \cup [|m|, +\infty)$ is purely continuous spectrum of $A_\eta$.
For this purpose let $S$ be the operator from~\eqref{def_S} and define the operator $T := S^* \upharpoonright (\dom A_\eta + \dom A_0)$. Let $\Gamma_0, \Gamma_1$ be the mappings from~\eqref{bt} and define $\Gamma_0^\Theta, \Gamma_1^\Theta: \dom T \rightarrow L^2(\Sigma; \mathbb{C}^2)$ by
\begin{equation*}
  \Gamma_0^\Theta f := \Gamma_1 f - \Theta \Gamma_0 f, \quad \Gamma_1^\Theta f := -\Gamma_0 f, \quad f \in \dom T.
\end{equation*}
Then one verifies in a similar way as in \cite[Proposition~2.4]{BHM20} that $\{L^2(\Sigma; \mathbb{C}^2), \Gamma_0^\Theta, \Gamma_1^\Theta \}$ is a quasi boundary triple for $S^*$ in the sense of \cite{BL07} with $T \upharpoonright \ker \Gamma_0^\Theta = A_\eta$ and Weyl function
\begin{equation*}
  \rho(A_\eta) \ni z \mapsto -\big(\Theta - M(z)\big)^{-1}.
\end{equation*}
Note that the operator $S$ in~\eqref{def_S} is simple; this can be shown as in \cite[Proposition~3.2]{BHM20} using \cite[Step 2 in the proof of Theorem 3.4]{BR20}. Hence,  Lemma~\ref{Proposition_continuous_spectrum}~(i) and \cite[Theorem~3.2]{BR15_2} imply that $(-\infty, -|m|] \cup [|m|, +\infty) \cap \sigma_\text{p}(A_\eta)=\emptyset$. Furthermore, Lemma~\ref{Proposition_continuous_spectrum}~(ii) and \cite[Theorem~3.5]{BR15_2} yield $(-\infty, -|m|) \cup (|m|, +\infty) \subset \sigma_\text{c}(A_\eta)$. Combining these facts shows
that $(-\infty, -|m|] \cup [|m|, +\infty)$ is contained in the purely continuous spectrum of $A_\eta$.
This finishes the proof of Theorem~\ref{main}.
\end{proof}



\end{document}